\documentclass{amsart}
\usepackage{amssymb, enumerate}
\usepackage{amsrefs}
\usepackage{graphicx}
\usepackage{tikz-cd}
\usepackage[percent]{overpic}
\usepackage{mathtools}
\usepackage{comment}
\usepackage{enumitem}
\usepackage{caption}
\usepackage{subcaption}

\usepackage{hyperref}

\newcommand{\br}{\overline}
\newcommand{\R}{\mathbb R}
\newcommand{\C}{\mathbb C}
\newcommand{\D}{\mathbb D}

\newcommand{\N}{\mathbb N}

\theoremstyle{plain}
\newtheorem{theorem}{Theorem}
\newtheorem{lemma}[theorem]{Lemma}

\newtheorem{prop}[theorem]{Proposition}
\newtheorem{corollary}[theorem]{Corollary}

\newtheoremstyle{case}{}{}{}{}{}{:}{ }{}
\theoremstyle{case}
\newtheorem{case}{Case}

\theoremstyle{definition}

\theoremstyle{remark}
\newtheorem{remark}[theorem]{Remark}

\DeclareMathOperator{\J}{\mathcal{J}}
\DeclareMathOperator{\dist}{\textup{\text{dist}}}
\DeclareMathOperator{\diam}{\textup{\text{diam}}}

\DeclareMathOperator{\Cdim}{\textup{Cdim}}

\DeclareMathOperator{\md}{\textup{mod}}

\numberwithin{equation}{section}
\numberwithin{theorem}{section}

\begin{document}

\title{Semi-hyperbolic rational maps and size of fatou components}
\author{Dimitrios Ntalampekos}
\thanks{The author was partially supported by NSF grant DMS-1506099}
\address{Department of Mathematics\\ University of California, Los Angeles\\
CA 90095, USA}
\email{dimitrisnt@math.ucla.edu}

\date{\today}

\begin{abstract} Recently, Merenkov and Sabitova introduced the notion of a homogeneous planar set. Using this notion they proved a result for Sierpi\'nski carpet Julia sets of hyperbolic rational maps that relates the diameters of the peripheral circles to the Hausdorff dimension of the Julia set. We extend this theorem to Julia sets (not necessarily Sierpi\'nski carpets) of semi-hyperbolic rational maps, and prove a stronger version of the theorem that was conjectured by Merenkov and Sabitova.
\end{abstract}

\maketitle

\section{Introduction}
In this paper we establish a relation between the size of the Fatou components of a semi-hyperbolic rational map and the Hausdorff dimension of the Julia set. Before formulating the results, we first discuss some background.

A rational map $f:\widehat \C \to \widehat \C$ of degree at least $2$ is \textit{semi-hyperbolic} if it has no parabolic cycles, and all critical points in its Julia set $\J(f)$ are \textit{non-recurrent}. We say that a point $x$ is \textit{non-recurrent} if $x\notin \omega(x)$, where $\omega(x)$ is the set of accumulation points of the orbit $\{f^n(x)\}_{n\in \N}$ of $x$. 

In our setting, we require that the Julia set $\J(f)$ is connected and that there are infinitely many Fatou components. Let $\{D_k\}_{k\geq 0}$ be the sequence of Fatou components, and define $C_k\coloneqq \partial D_k$. Since $\J(f)$ is connected, it follows that each component $D_k$ is simply connected, and thus $C_k$ is connected.  

We say that the collection $\{C_k\}_{k\geq 0}$ is a \textit{packing} $\mathcal P$ and we define the \textit{curvature distribution function} associated to $\mathcal P$ (see below for motivation of this terminology) by
\begin{align}\label{N(x) -Definition}
	N(x)= \# \{ k : (\diam C_k)^{-1}\leq x\}
\end{align}
for $x>0$. Here $\#A$ denotes the number of elements in a given set $A$. Also, the \textit{exponent} $E$ of the packing $\mathcal P$ is defined by
\begin{align}\label{E - Definition}
	E= \inf \biggl\{ t\in \R : \sum_{k\geq 0} (\diam C_k)^t <\infty \biggr\},
\end{align}     
where all diameters are in the spherical metric of $\widehat \C$.

In the following, we write $a\simeq b$ if there exists a constant $C>0$ such that $\frac{1}{C}a\leq b\leq Ca$. If only one of these inequalities is true, we write $a\lesssim b$ or $b\lesssim a$ respectively. We denote the Hausdorff dimension of a set $J\subset \widehat \C$ by $\dim_H J$ (see Section \ref{Section - Minkowski}). We now state our main result.

\begin{theorem}\label{Theorem - Main}
Let $f:\widehat \C \to \widehat \C$ be a semi-hyperbolic rational map such that the Julia set $\J(f)$ is connected and the Fatou set has infinitely many components. Then 
\begin{align*}
	0<\liminf_{x\to\infty} \frac{ N(x)}{x^s} \leq \limsup_{x\to\infty} \frac{N(x)}{x^s} <\infty,
\end{align*}
where $N$ is the curvature distribution function of the packing of the Fatou components of $f$ and $s=\dim_H \J(f)$. In particular $N(x) \simeq x^s$.
\end{theorem}

It is remarkable that the curvature distribution function has polynomial growth. As a consequence, we have the following corollary.

\begin{corollary}\label{Corollary} Under the assumptions of Theorem \ref{Theorem - Main} we have
\begin{align*}
	\lim_{x\to\infty} \frac{\log N(x)}{\log x} = E= \dim_H \J(f),
\end{align*}
where $N$ is the curvature distribution function, and $E$ is the exponent of the packing of the Fatou components of $f$.
\end{corollary}

This essentially says that one can compute the Hausdorff dimension of the Julia set just by looking at the diameters of the (countably many) Fatou components, which lie in the complement of the Julia set. 

The study of the curvature distribution function and the terminology is motivated by the \textit{Apollonian circle packings}.

An \textit{Apollonian circle packing} is constructed inductively as follows. Let $C_1,C_2,C_3$ be three mutually tangent circles in the plane with disjoint interiors. Then by a theorem of Apollonius there exist exactly two circles that are tangent to all three of $C_1,C_2,C_3$. We denote by $C_0$ the outer circle that is tangent to $C_1,C_2,C_3$ (see Figure \ref{fig:apollonian}). For the inductive step we apply Apollonius's theorem to all triples of  mutually tangent circles of the previous step. In this way, we obtain a countable collection of circles $\{C_k\}_{k\geq 0}$. We denote by $\mathcal P= \{C_k\}_{k\geq 0}$ the Apollonian circle packing constructed this way. If $r_k$ denotes the radius of $C_k$, then $r_k^{-1}$ is the curvature of $C_k$. The curvatures of the circles in Apollonian packings are of great interest in number theory because of the fact that if the four initial circles $C_0,C_1,C_2,C_3$ have integer curvatures, then so do all the rest of the circles in the packing. Another interesting fact is that if, in addition, the curvatures of all circles  in the packing share no common factor greater than one, then there are infinitely many circles in the packing with curvature being a prime number. For a survey on the topic see \cite{Oh}.

In order to study the curvatures of an Apollonian packing $\mathcal P$ one defines the \textit{exponent} $E$ of the packing by
\begin{align*}
E= \inf \biggl\{ t\in \R : \sum_{k\geq 0} r_k^t <\infty \biggr\}
\end{align*}
and the \textit{curvature distribution function} associated to $ \mathcal P$ by
\begin{align*}
N(x)= \# \{ k : r_k^{-1}\leq x\}
\end{align*}
for $x>0$. We remark here that the radii $r_k$ are measured with the Euclidean metric of the plane, in contrast to \eqref{N(x) -Definition} where we use the spherical metric. Let $D_k$ be the open ball enclosed by $C_k$. The \textit{residual set} $\mathcal S$ of a packing $\mathcal P$ is defined by $\br {D_0} \setminus \bigcup_{k\geq 1} D_k$. The set $\mathcal S$ has fractal nature and its Hausdorff dimension $s=\dim_H \mathcal S$ is related to $N(x)$ and $E$ by the following result of Boyd.
\begin{theorem}[{{\cite{Bo1}, \cite{Bo2}}}] \label{Theorem - Boyd}
If $\mathcal P$ is an Apollonian circle packing, then
\begin{align*}
\lim_{x\to\infty} \frac{\log N(x)}{\log x}=E=\dim_H\mathcal S.
\end{align*} 
\end{theorem}
Recently, Kontorovich and Oh proved the following stronger version of this theorem:
\begin{theorem}[{{\cite[Theorem 1.1]{KO}}}] \label{Theorem - Kontor-Oh}
If $\mathcal P$ is an Apollonian circle packing, then
\begin{align*}
\lim_{x\to\infty} \frac{N(x)}{x^s} \in (0,\infty),
\end{align*}
where $s= E = \dim_H \mathcal S$. In particular, $N(x)\simeq x^s$.
\end{theorem}

\begin{figure}
	\begin{overpic}[width=0.5\textwidth]{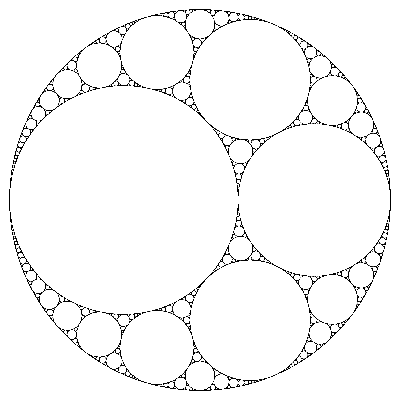}
 		\put (82,82) {$C_0$}
 		\put (55,85) {$C_1$}
 		\put (30,68) {$C_2$}
 		\put (75,60) {$C_3$}
	\end{overpic}
	\captionof{figure}{An Apollonian circle packing.}
  	\label{fig:apollonian}
\end{figure}
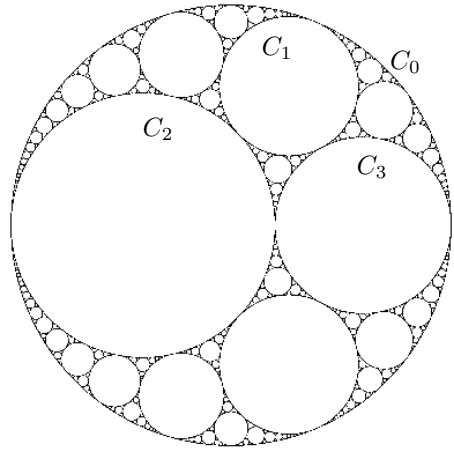

In \cite{MS}, Merenkov and Sabitova observed that the curvature distribution function $N(x)$ can be defined also for other planar fractal sets such as the Sierpi\'nski gasket and  Sierpi\'nski carpets. More precisely, if $\{C_k\}_{k\geq 0}$ is a collection of topological circles in the plane, and $D_k$ is the open topological disk enclosed by $C_k$, such that $ {D_0}$ contains $C_k$ for $k\geq 1$, and $D_k$ are disjoint for $k\geq 1$, one can define the \textit{residual set } $\mathcal S$ of the \textit{packing} $\mathcal P= \{C_k\}_{k\geq 0}$ by $\mathcal S= \br {D_0} \setminus  \bigcup_{k\geq 1} D_k$. A fundamental result of Whyburn implies that if the disks $\br {D_k}$, $k\geq 1$ are disjoint with $\diam D_k\to 0$ as $k\to\infty$ and $\mathcal S$ has empty interior, then $\mathcal S$ is homeomorphic to the standard Sierpi\'nski carpet \cite{Wh}. In the latter case we say that $\mathcal S$ is a \textit{Sierpi\'nski carpet} (see Figure \ref{fig:test2} for a Sierpi\'nski carpet Julia set). One can define the curvature of a topological circle $C_k$ as $(\diam C_k)^{-1}$. Then the \textit{curvature distribution function} associated to $\mathcal P$ is defined as in \eqref{N(x) -Definition} by $N(x)= \#\{k: (\diam C_k)^{-1} \leq x\}$ for $x>0$. Similarly, the exponent $E$ of $\mathcal P$ is defined as in \eqref{E - Definition}. 

In general, the limit $\lim_{x\to\infty} {\log N(x)}/{\log x}$ does not exist, but if we impose further restrictions on the geometry of the circles $C_k$, then we can draw conclusions about the limit. To this end, Merenkov and Sabitova introduced the notion of \textit{homogeneous planar sets} (see Section \ref{Section - homogeneous} for the definition). However, even these strong geometric restrictions are not enough to guarantee the existence of the limit. The following theorem hints that a self-similarity condition on $\mathcal S$ would be sufficient for our purposes.

\begin{theorem}[{{\cite[Theorem 6]{MS}}}]\label{Theorem - Merenkov-Sabitova} Assume that $f$ is a hyperbolic rational map whose Julia set $\J(f)$ is a Sierpi\'nski carpet. Then
\begin{align*}
\lim_{x\to\infty} \frac{\log N(x)}{\log x}= E=\dim_H \J(f),
\end{align*}
where $N$ is the curvature distribution function and $E$ is the exponent of the packing of the Fatou components of $f$.
\end{theorem} 
The authors made the conjecture that for such Julia sets we actually have an analogue of Theorem \ref{Theorem - Kontor-Oh}, namely $\lim_{x\to\infty}N(x)/x^s \in (0,\infty)$, where $s=\dim_H \J(f)$. Note that Theorem \ref{Theorem - Main} partially addresses the issue by asserting that $N(x)\simeq x^s$. However, we believe that the limit $\lim_{x\to\infty}N(x)/x^s $ does not exist in general for Julia sets. Observe that the conclusion of Theorem \ref{Theorem - Main} remains valid if we alter the metric that we are using in the definition of $N(x)$ in a bi-Lipschitz way. For example, if the Julia set $\J(f)$ is contained in the unit disk of the plane we can use the Euclidean metric instead of the spherical. On the other hand, the limit of $N(x)/x^s$ as $x\to\infty$ is much more sensitive to changes of the metric. The following simple example of the \textit{standard Sierpi\'nski carpet} provides some evidence that the limit will not exist even for packings with very ``nice" geometry. 

The standard Sierpi\'nski carpet is constructed as follows. We first subdivide the unit square $[0,1]^2$ into $9$ squares of equal size and then remove the interior of the middle square. We continue subdividing each of the remaining $8$ squares into $9$ squares, and proceed inductively. The resulting set $\mathcal S$ is the standard Sierpi\'nski carpet and its Hausdorff dimension is $s= \log8 /\log 3$. The set $\mathcal S$ can be viewed as the residual set of a packing $\mathcal P= \{C_k\}_{k\geq 0}$, where $C_0$ is the boundary of the unit square, and $C_k,k\geq 1$ are the boundaries of the squares that we remove in each step in the construction of $\mathcal S$. Using the Euclidean metric, note that for each $n\in \N$ the quantity $N(3^n/ \sqrt{2})$ is by definition the number of curves $C_k$ that have diameter at least $\sqrt{2}/3^n$. Thus, 
$$N(3^n/ \sqrt{2})= 1+ 1+8^1+8^2 \dots +8^{n-1}= 1+ \frac{8^n-1}{7}$$
(note that we also count $C_0$). Since $3^{n\cdot s}= 8^n$, we have
\begin{align*}
\lim_{n\to\infty} \frac{N(3^n/\sqrt{2})}{ (3^{n}/\sqrt{2})^s }= \frac{\sqrt{2^s}}{7}.
\end{align*}
On the other hand, it is easy to see that $N(3^n/\sqrt{2})= N( 3^n)$, since there are no curves $C_k$ with diameter in the interval $[\frac{1}{3^n},\frac{\sqrt{2}}{3^n})$. Thus, $\lim_{n\to\infty} N(3^n)/ (3^n)^s= 1/7$, and this shows that $\lim_{n\to\infty} N(x)/ x^s$ does not exist. In general, if one can show that there exists some constant $c>0, c\neq 1$ such that $N(x)=N(cx)$ for large $x$, then the limit will not exist.

\begin{figure}
\centering
\begin{minipage}{.5\textwidth}
  \centering
  \includegraphics[width=.6\linewidth]{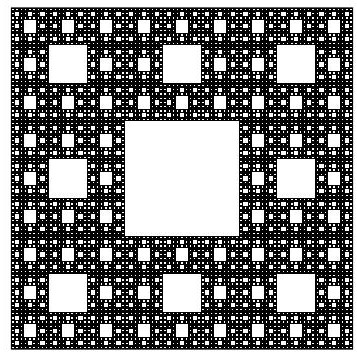}
  \captionof{figure}{The standard Sierpi\'nski carpet}
  \label{fig:test1}
\end{minipage}%
\begin{minipage}{.5\textwidth}
  \centering
  \begin{overpic}[width=0.6\textwidth]{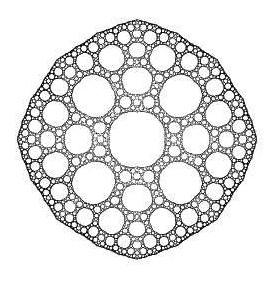}
 	\put (70,90) {$\C\setminus D_0$}
 	\put (44,48) {$D_k$}
  \end{overpic}
  \captionof{figure}{The Julia set of $f(z)=z^2-\frac{1}{16z^2}$}
  \label{fig:test2}
\end{minipage}
\end{figure}

We also note that in Theorem \ref{Theorem - Main} one might be able to weaken the assumption that $f$ is semi-hyperbolic, but the assumption that $f$ has connected Julia set is necessary, since there exist rational maps whose Fatou components (except for two of them) are nested annuli, and in fact in this case there exist infinitely many Fatou components with ``large" diameters (see \cite[Proposition 7.2]{Mc}). Thus, if $N(x)$ is the number of Fatou components whose diameter is at least $1/x$, we would have $N(x)=\infty$ for large $x$.

The proof of Theorem \ref{Theorem - Main} will be given in two main steps. In Section \ref{Section - Minkowski}, using the self-similarity of the Julia set we will establish relations between the Hausdorff dimension of the Julia set and its {Minkowski dimension} (see Section \ref{Section - Minkowski} for the definition). Then in Section \ref{Section - homogeneous} we will observe that the Julia sets of semi-hyperbolic maps are homogeneous sets, satisfying certain geometric conditions (see Section \ref{Section - homogeneous} for the definition). These conditions allow one to relate the quantity $N(x)/x^s$ with the {Minkowski content} of the Julia set. Using these relations, and the results of Section \ref{Section - Minkowski}, the proof of Theorem \ref{Theorem - Main} will be completed.

Before proceeding to the above steps, we need some important distortion estimates for semi-hyperbolic rational maps that we establish in Section \ref{Section - Conformal elevator}, and we will refer to them as the {Conformal Elevator}. These are the key estimates that we will use in establishing geometric properties of the Julia set. Similar estimates have been established for sub-hyperbolic rational maps in \cite[Lemma 4.1]{BLM}.

\subsection*{Acknowledgements.} The author would like to thank his advisor, Mario Bonk, for many useful comments and suggestions, and for his patient guidance. He also thanks the anonymous referees for their careful reading of the manuscript and their thoughtful comments.

\section{Conformal elevator for semi-hyperbolic maps}\label{Section - Conformal elevator}

The heart of this section is Lemma \ref{Lemma - elevator} and the whole section is devoted to proving it.

Let $f:\widehat \C\to \widehat \C$ be a semi-hyperbolic map with $\J (f) \neq \widehat \C$; in particular, by Sullivan's classification and the fact that semi-hyperbolic rational maps have neither parabolic cycles (by definition) nor Siegel disks and Herman rings (\cite[Corollary]{Ma}), $f$ must have an attracting or superattracting periodic point. Conjugating $f$ by a rotation of the sphere $\widehat \C$, we may assume that $\infty$ is a periodic point in the Fatou set. Furthermore, conjugating again with a Euclidean similarity, we can achieve that $\J(f) \subset \frac{1}{2}\D$, where $\D$ denotes the unit disk in the plane. Note that these operations do not affect the conclusion of Theorem \ref{Theorem - Main}, since a rotation is an isometry in the spherical metric that we used in the definition of $N(x)$, and a scaling only changes the limits by a factor. Furthermore, since the boundaries $C_k$ of the Fatou components $D_k$ have been moved away from $\infty$, the diameters of $C_k$ in spherical metric are comparable to the diameters in the Euclidean metric. This easily implies that the conclusion of Theorem \ref{Theorem - Main} is not affected if we define $N(x)=\#\{k: (\diam C_k)^{-1}\leq x\}$ using instead the Euclidean metric for measuring the diameters.

In this section the Euclidean metric will be used in all of our considerations.

By semi-hyperbolicity (see \cite[Theorem II(b)]{Ma}) and compactness of $\J(f)$, there exists $\varepsilon_0>0$ such that for every $x\in \J(f)$ and for every connected component $W$ of $f^{-n}(B(x,\varepsilon_0))$ the degree of $f^n :W \to B(x,\varepsilon_0)$ is bounded by some fixed constant $D_0>0$ that does not depend on $x, W, n$. Furthermore, we can choose an even smaller $\varepsilon_0$ so that the open $\varepsilon_0$-neighborhood of $\J(f)$ that we denote by $N_{\varepsilon_0}(\J(f))$ is contained in $\D$, and avoids the poles of $f$ that must lie in the Fatou set. Then $f$ is uniformly continuous in $N_{\varepsilon_0/2} (\J(f))$ in the Euclidean metric, and in particular, there exists $\delta_0>0$ such that for any $U \subset N_{\varepsilon_0/2} (\J(f))$ with $\diam U < \delta_0$ we have $\diam f(U) <\varepsilon_0/2$. 

Let $p\in \J(f)$, $0<r\leq \delta_0/2$ be arbitrary, and define $B\coloneqq B(p,r)$. Since for large $N\in\N$ we have $f^N(B) \supset \J(f)$ (e.g. see \cite[Corollary 14.2]{Mil}), there exists a largest $n\in \N$ such that $\diam f^n(B) <\varepsilon_0/2 $. By the choice of $n$, we have $\diam f^{n+1}(B)\geq \varepsilon_0/2$. Using the uniform continuity and the choice of $\delta_0$, it follows that $\diam f^n(B) \geq \delta_0$, thus  
\begin{align}\label{diamf^n(B) bounds}
\delta_0 \leq  \diam f^n(B) \leq \varepsilon_0/2.
\end{align} 

We now state the main lemma.
\begin{lemma}\label{Lemma - elevator} There exist constants $\gamma,r_1, K_1,K_2>0$ independent of $B=B(p,r)$ (and thus of $n$) such that:
\begin{enumerate}[label=\emph{(\alph*)}]
\item If $A\subset B$ is a connected set, then
\begin{align*}
\frac{\diam A}{\diam B} \leq K_1  (\diam f^n(A))^\gamma.
\end{align*}
\item $B( f^n(p),r_1) \subset f^n(B(p,r/2) ).$
\item For all $u,v\in B$ we have
\begin{align*}
|f^n(u)-f^n(v)| \leq K_2 \frac{|u-v|}{\diam B}.
\end{align*}

\end{enumerate}
\end{lemma}

This lemma asserts that any ball of small radius centered at the Julia set can be blown up to a certain size, using some iterate $f^n$, with good distortion estimates. For \textit{hyperbolic} rational maps (i.e., no parabolic cycles and no critical points on the Julia set) the map $f^n$ would actually be bi-Lipschitz and part (c) of the above lemma would be true with $\simeq$ instead of $\leq$. However, in the semi-hyperbolic case, the presence of critical points on the Julia set prevents such good estimates, but part (a) of the lemma restores some of them. 
 
In order to prepare for the proof we need some distortion lemmas. Using Koebe's distortion theorem (e.g., see \cite[Theorem 1.3]{Po}) one can derive the following lemma.
\begin{lemma}\label{Lemma - Koebe}
Let $g: \D \to \C$ be a univalent map and let $0<\rho<1$. Then there exists a constant $C_\rho>0$ that depends only on $\rho$, such that  
\begin{align*}
\frac{1}{C_\rho}  |g'(0)| |u-v| \leq |g(u)-g(v)| \leq  C_\rho  |g'(0)| |u-v| 
\end{align*} 
for all $u,v \in B(0,\rho)$. 
\end{lemma} 

We will be using the notation $|g(u)-g(v)| \simeq_\rho  |g'(0)| |u-v|$. We also need the next lemma.

\begin{lemma}\label{Lemma - Simply connected components} Let $g:\widehat \C\to \widehat \C$ be a semi-hyperbolic rational map with $\J(g)\neq \widehat \C$ and assume that $\J(g)$ is connected. Then there exists $\varepsilon>0$ such that for all $x\in \J(g)$, each component of $g^{-m}(B(x,\varepsilon))$ is simply connected, for all $m\in \N$.
\end{lemma}
\begin{proof}
As before, by conjugating, we may assume that $\infty$ is a periodic point in the Fatou set, and the Julia set is ``far" from the poles of $g$. By semi-hyperbolicity (see \cite[Theorem II(c)]{Ma}), for each $x\in \J(g)$ and $\eta>0$, there exists $\varepsilon>0$ such that each component of $g^{-m}(B(x,\varepsilon))$ has Euclidean diameter less than $\eta$, for all $m\in \N$. By compactness of $\J(g)$, we may take $\varepsilon>0$ to be uniform in $x$. We choose a sufficiently small $\eta$ such that the $3\eta$-neighborhood $N_{3\eta} (\J(g))$ of $\J(g)$ does not contain any poles of $g$.

We claim that each component of $g^{-m}(B(x,\varepsilon))$ is simply connected. If this was not the case, there would exist an open component $W$ of $g^{-{m_0}}(B(x,\varepsilon))$, and a non-empty family of compact components $\{V_i\}_{i\in I}$ of $ \C \setminus W$. Thus $\diam V_i \leq \diam W <\eta$ for $i\in I$. Assume that $m_0\in \N$ is the smallest such integer. Note that $W$ intersects the Julia set $\J(g)$, because $g^{m_0}(W)=B(x,\varepsilon)$ does so. Hence, we have $W\subset N_{\eta}(\J(g))$. Since $V_i,i\in I$ and $W$ share at least one common boundary point, it follows that $V_i\subset N_{2\eta}(\J(g))$, and in particular $V_i$ does not contain any poles of $g$, i.e., $\infty \notin g( V_i)$ for all $i\in I$.

By the choice of $m_0$ the set $g(W) \subset g^{-m_0+1} (B(x,\varepsilon)) $ is a simply connected set in the $\eta$-neighborhood of $\J(g)$. Note that $\bigcup_{i\in I}g(V_i)$ cannot be entirely contained in $\br{g(W)}$, otherwise $W$ would not be a component of $g^{-m_0}(B(x,\varepsilon))$. Thus, there exists some $V_i\eqqcolon V$ and a point $w_0 \in (\widehat \C \setminus \br{ g(W)})\cap g(V)$. We connect the point $w_0$ to $\infty$ with a path $\gamma \subset \widehat \C \setminus \br {g(W)}$, and then we lift $\gamma$ under $g$ to a path $\alpha\subset \widehat \C$ that connects a preimage $z_0 \in V$ of $w_0$ to a pole of $g$ (see \cite[Lemma A.16]{BM} for path-lifting under branched covers). The path $\alpha$ cannot intersect $W$, so it stays entirely in $ V$. This contradicts the fact that $ V$ contains no poles.
\end{proof}

Now we are ready to start the proof of Lemma \ref{Lemma - elevator}. Since $\diam f^n(B)<\varepsilon_0/2$, for $x=f^n(p)\in \J(f)$ we have $f^n(B)\subset B(x,\varepsilon_0/2)$, and for the component $\Omega$ of $f^{-n}(B(x,\varepsilon_0))$ that contains $B$ we have that the degree of $f^n: \Omega \to B(x,\varepsilon_0)$ is bounded by $D_0$. Lemma \ref{Lemma - Simply connected components} implies that we can refine our choice of $\varepsilon_0$ such that $\Omega$ is also simply connected. 

Let $\psi: \Omega\to \D$ be the Riemann map that maps the center $p$ of $B$ to $0$, and $\phi :B(x,\varepsilon_0) \to \D$ be the translation of $x$ to $0$, followed by a scaling by $1/\varepsilon_0$, so we obtain the following diagram:
\begin{equation}\label{Commutative diagram}
\begin{tikzcd}
\Omega \arrow{r}{f^n} \arrow[swap]{d}{\psi} & B(x,\varepsilon_0) \arrow{d}{\phi} \\
\D  \arrow{r}{} & \D
\end{tikzcd}
\end{equation}

\begin{figure}
	\begin{overpic}[width=.75\linewidth]{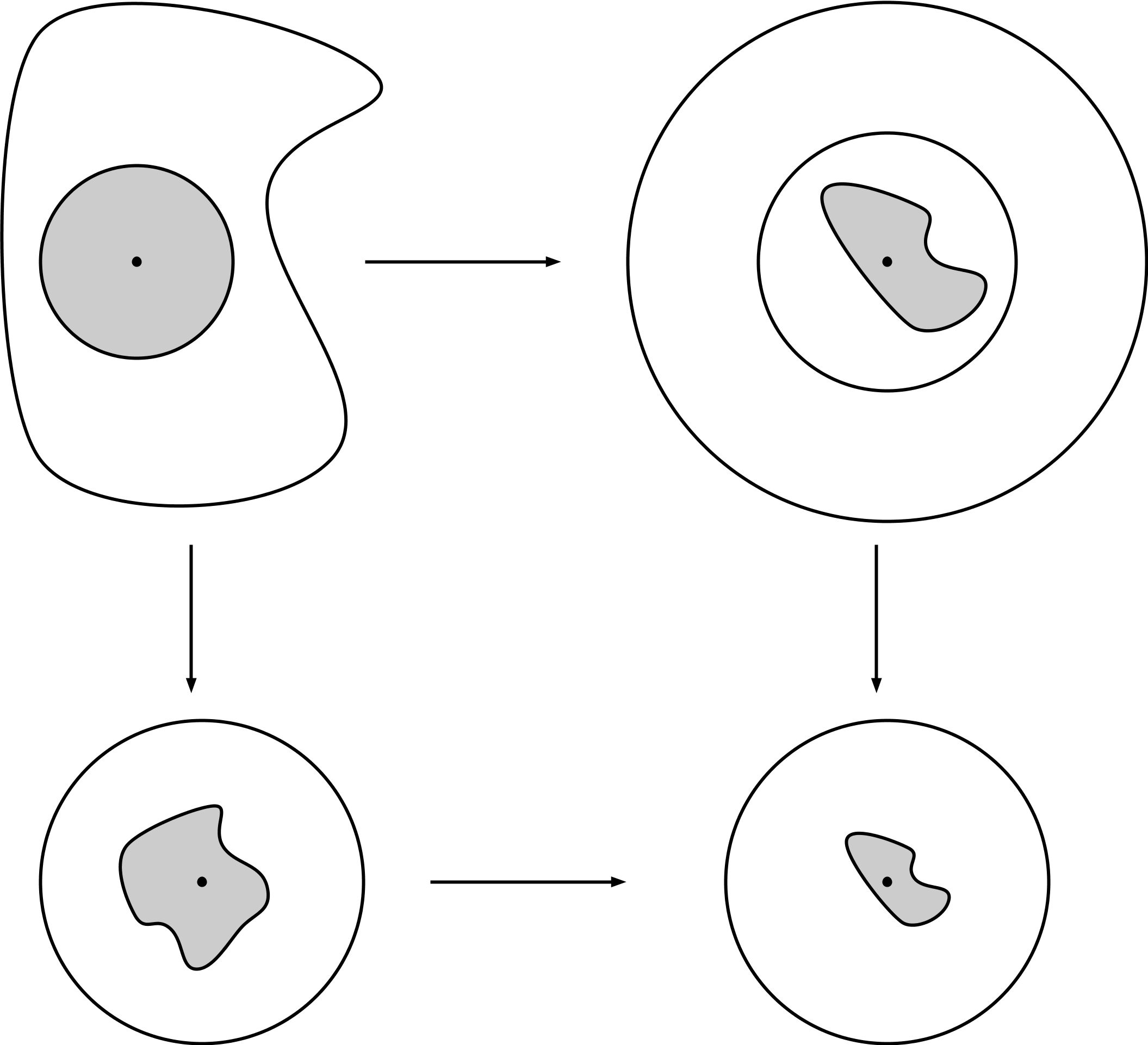}
 	\put (40,70) {$f^n$}
 	\put (13,35) {$\psi$}
 	\put (80,35) {$\phi$}
 	\put (30,45) {$\Omega$}
 	\put (20,60) {$B$}
 	\put (91,45) {$B(x,\varepsilon_0)$}
 	\put (72,59) {$f^n(B)$}
 	\put (10,70) {$p$}
 	\put (75,70) {$x$}
 	\put (18.5,5) {$\psi(B)$}
 	\put (69,5) {$\phi(f^n(B))$}
  \end{overpic}
  \captionof{figure}{The commutative diagram \eqref{Commutative diagram}.}
  \label{fig:diagram}
\end{figure}
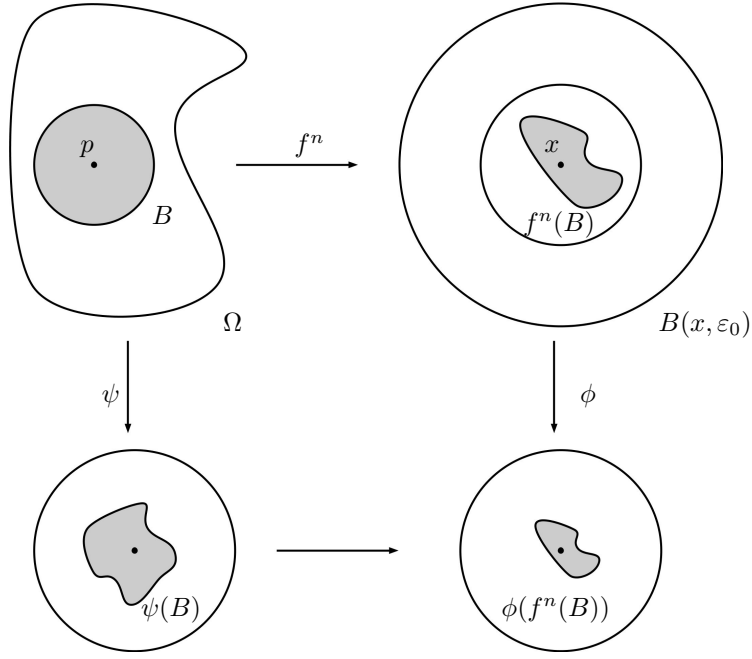

The proof will be done in several steps. First we prove that $\psi(B)$ is contained in a ball of fixed radius smaller than $1$. Second, we show a distortion estimate for $\psi$, namely it is roughly a scaling by $1/\diam B$. In the end, we complete the proofs of (a),(b),(c), using lemmas that are generally true for proper maps.

 We claim that there exists $\rho>0$, independent of $B$ such that 
\begin{align}\label{existence of rho}
\psi(B) \subset B(0,\rho)\subset \D.
\end{align}
This will be derived from the following modulus distortion lemma. We include first some definitions.

If $\Gamma$ is a family of curves in $\C$, we define the \textit{modulus of $\Gamma$}, denoted by $\md(\Gamma)$, as follows. A function $\rho:\C \to [0,\infty]$ is called \textit{admissible} for $\md(\Gamma)$ if 
\begin{align*}
\int_\gamma \rho ds \geq 1
\end{align*}
for all curves $\gamma\in \Gamma$. Then 
\begin{align*}
\md(\Gamma)\coloneqq \inf_{\rho} \int_{\C} \rho(z)^2 dm_2(z),
\end{align*}
where $m_2$ denotes the 2-dimensional Lebesgue measure, and the infimum is taken over all admissible functions. The modulus has the \textit{monotonicity} property, namely if $\Gamma_1,\Gamma_2$ are path families and $\Gamma_1\subset \Gamma_2$, then
\begin{align*}
\md(\Gamma_1)\leq \md(\Gamma_2).
\end{align*}
Another important property of modulus is \textit{conformal invariance}: if $\Gamma$ is a curve family in an open set $U\subset \C$ and $g:U\to V$ is conformal, then 
\begin{align*}
\md(\Gamma)=\md(g(\Gamma)).
\end{align*}
We direct the reader to \cite[pp.~132--133]{LV} for more background on modulus.

If $U$ is a simply connected region, and $V$ is a connected subset of $U$ with  $\br V\subset U$, we denote by $\md(U\setminus \br V)$ the modulus of the curve family that separates $V$ from $\C\setminus \br U$.

\begin{lemma}\label{Lemma - modulus} Let $U,U' \subset \C$ be simply connected regions, and $g:U\to U'$ be a proper holomorphic map of degree $D$.
\begin{enumerate}[label=\emph{(\alph*)}]
\item If $V'$ is a Jordan region with $V'\subset \br {V'} \subset U'$, and $V$ is a component of $g^{-1}(V')$, then
\begin{align*}
\md(U'\setminus \br{V'}) \leq D \md(U\setminus \br V).
\end{align*}

\item If $V$ is a Jordan region with $V\subset \br V\subset U$, and $V'=g(V)$, then
\begin{align*}
\md( U\setminus \br V) \leq \md (U'\setminus \br{V'} ).
\end{align*}

\end{enumerate}
\end{lemma}

A particular case of this lemma is \cite[Lemma 5.5]{Mc2}, but we include a proof of the general statement since we were not able to find it in the literature.
\begin{proof}
Using the conformal invariance of modulus we may assume that $U$ and $U'$ are bounded Jordan regions.

We first show $\textrm{(a)}$. Using a conformal map, we map the annulus $U'\setminus \br{V'}$ to the circular annulus $\D \setminus \br {B}(0,r)$, and by composing with $g$, we assume that we have a proper holomorphic map $g: U\setminus \br V \to \D \setminus \br {B}(0,r)$, of degree at most $D$. We divide the annulus $\D \setminus \br {B}(0,r)$ into nested circular annuli centered at the origin $A_1',\dots,A_k',k\leq D$ such that each $A_i'$ does not contain any critical value of $g$ in its interior. Note that 
\begin{align*}
\frac{1}{2\pi} \log (1/r)=\md (\D \setminus \br {B}(0,r))= \sum_{i=1}^k \md (A_i'),
\end{align*}   
where we denote by $\md(A_i')$ the modulus of curves that separate the complementary components of the annulus $A_i'$. We fix $\varepsilon>0$. By making the annuli $A_i'$ a bit thinner, we can achieve that $\partial A_i'$ does not contain any critical value of $g$, and
\begin{align}\label{Lemma - Proper Modulus - epsilon}
\sum_{i=1}^k \md A_i' \geq \md (\D \setminus \br {B}(0,r))-\varepsilon.
\end{align}
Let $A_i$ be a preimage of $A_i'$, so that $A_1,\dots,A_k$ are nested annuli separating $V$ from $\C \setminus \br U$, and avoiding the critical points of $g$. Note that $g:A_i\to A_i'$ is a covering map of degree $d_i \leq D$, thus $\md A_i = \md A_i' /d_i \geq  \md A_i' /D$. This implies that 
\begin{align}\label{Lemma - Proper Modulus - sum}
\md (U\setminus \br V) \geq  \sum_{i=1}^k {\md (A_i)} \geq \sum_{i=1}^k {\md (A_i')} /D.
\end{align}
To see the first inequality, note that an admissible function $\rho$ for $\md (U\setminus \br V)$ yields admissible functions $\rho|_{A_i}$ for $\md (A_i)$. Combining \eqref{Lemma - Proper Modulus - sum} and \eqref{Lemma - Proper Modulus - epsilon} we obtain
\begin{align*}
\md(U\setminus \br V)   \geq   \frac{1}{D} \left(\md (\D \setminus \br {B}(0,r))-\varepsilon\right).
\end{align*}
Letting $\varepsilon \to 0$ one concludes the proof.

The inequality in $\textrm{(b)}$ follows from Poletski\u{\i}'s inequality \cite[Chapter II, Section 8]{Ri}. Since holomorphic maps are \textit{$1$-quasiregular} (see \cite[Chapter I]{Ri} for definition and background), we have 
\begin{align}\label{Lemma -  Poletski}
\md (g(\Gamma)) \leq \md (\Gamma)
\end{align}
for all path families $\Gamma$ in $U$. First we shrink the regions $U$ and $U'$ as follows. Consider a Jordan curve $\gamma_1'$ very close to $\partial U'$ such that $\gamma_1'$ encloses a region $U_1'$ that contains $V'$ and all critical values of $g$. Then $U_1\coloneqq g^{-1}(U_1')$ is a Jordan region that contains $V$ and all critical points of $g$. 

Let $\Gamma $ be the family of paths in $\br {U_1}\setminus V$ that connect $\partial V$ to $ \partial U_1$ and avoid preimages of critical values of $g$, which are finitely many. Also, note that $g(\Gamma) \supset \Gamma'$, where $\Gamma'$ is the family of paths in $\br {U'_1} \setminus V'$ that connect $\partial V'$ to $\partial U_1'$, and avoid the critical values of $g$. To see this, observe that any such path $\gamma '$ has a lift $\gamma\subset \br{U_1}\setminus V$ that starts at $\partial V$ and ends at $\partial U_1$. 

Using monotonicity of modulus and \eqref{Lemma -  Poletski} we have $\md (\Gamma')\leq \md (g(\Gamma)) \leq \md (\Gamma)$. If $\tilde \Gamma$ is the family of all paths in $\br {U_1}\setminus V$ that connect $\partial V$ to $ \partial U_1$, then $\tilde \Gamma$ differs from $\Gamma$ by a family of zero modulus. The same is true for the corresponding family $\tilde \Gamma'$ in $\br {U_1'} \setminus V'$. Thus, we have $ \md( \tilde \Gamma' ) \leq \md(\tilde \Gamma)$. By reciprocality of the modulus and monotonicity, it follows that 
$$\md( U_1 \setminus \br V ) \leq \md( U_1' \setminus \br {V'}) \leq \md(U'\setminus \br{V}). $$

Finally, observe that the path family separating  $V$ from  $\C \setminus \br U$ can be written as an increasing union of families separating $V$ from sets of the form $\C \setminus \br {U_1}$, where $U_1$ gets closer and closer to $U$. Writing $\md (U\setminus \br V)$ as a limit of moduli of such families, one obtains the desired inequality.
\end{proof}

We now return to the proof of \eqref{existence of rho}.

Applying Lemma \ref{Lemma - modulus}$\textrm{(a)}$ to $g=f^n: \Omega \to B(x,\varepsilon_0)$, and using the fact that $f^n(B) \subset B(x, \varepsilon_0/2)$ along with monotonicity of modulus we obtain 
\begin{align*}
\frac{1}{2\pi} \log 2 = \md (B(x,\varepsilon_0)\setminus \br B(x,\varepsilon_0/2)) \leq D_0 \md( \Omega \setminus \br B).
\end{align*}
Since modulus is invariant under conformal maps, we have 
\begin{align} \label{Modulus estimate psi(B)}
\frac{1}{2\pi } \log2 \leq D_0 \md( \D \setminus \br {\psi (B)}).
\end{align}
If $ \zeta \in \br  {\psi(B)}$ is such that $|\zeta|= \sup \{|z| : z\in \br {\psi(B)}\}$ then by Gr\"otzsch's modulus theorem (see \cite[p.54]{LV}) we have $\md(\D \setminus \br {\psi(B)}) \leq \mu(|\zeta|)$, where $\mu : (0,1)\to (0,\infty)$ is a strictly decreasing bijection. Thus, by \eqref{Modulus estimate psi(B)} $\mu(|\zeta|)$ is uniformly bounded below, and by monotonicity there exists $\rho\in(0,1)$ such that $|\zeta| \leq \rho$. Hence, $\psi(B) \subset B(0,\rho)$, which proves \eqref{existence of rho}. 

Now, the version of Koebe's theorem in Lemma \ref{Lemma - Koebe} yields 
\begin{align}\label{Koebe for psi}
|u-v| \simeq_\rho |(\psi^{-1})'(0)| |\psi(u)-\psi(v)|
\end{align}
for all $u,v\in B$. We claim that $|(\psi^{-1})'(0)| \simeq \diam B$, so \eqref{Koebe for psi} can be rewritten as
\begin{align}\label{Koebe for psi-diamB}
|u-v| \simeq  \diam B |\psi(u)-\psi(v)|
\end{align}
for all $u,v\in B$. Using \eqref{Koebe for psi}, in order to prove our claim, it suffices to show that $\diam \psi(B) \simeq 1$. Note that by \eqref{diamf^n(B) bounds} we have $\diam f^n(B) \simeq 1$, and since $\phi$ is a scaling by a fixed factor, we have $\diam \phi (f^n(B)) \simeq 1$. Using Lemma \ref{Lemma - modulus}$\textrm{(b)}$ for the diagram \eqref{Commutative diagram}, and Gr\"otzsch's modulus theorem we obtain
\begin{align*}
\md( \D\setminus \br {\psi(B)}) \leq  \md( \D \setminus \br {\phi(f^n(B))})\leq \mu(|\zeta|),
\end{align*}
where $\zeta \in \br {\phi(f^n(B))}$ is the furthest point from the origin. Since $0,\zeta\in \br{\phi (f^n(B))}$, we have $|\zeta|  \simeq \diam \phi (f^n(B)) \simeq 1$. Monotonicity of $\mu$ implies that $\mu( |\zeta|) \lesssim 1$, thus
\begin{align*}
\md( \D\setminus \br {\psi(B)}) \lesssim 1.
\end{align*} 
On the other hand, if $\alpha \in \br {\psi(B)}$ is the furthest from the origin, we have $\diam \psi(B) \simeq |\alpha|$, and $B(0,|\alpha|) \supset \psi(B)$, thus by monotonicity of modulus
\begin{align*}
\md(\D \setminus \br B(0,|\alpha|) ) \leq \md (\D \setminus \br { \psi(B)}) \lesssim 1.
\end{align*}
This shows that $ \log (1/|\alpha|) \lesssim 1$, so $|\alpha|\gtrsim 1$, and this implies that $\diam \psi(B)\simeq 1$, as claimed.

Before proving part $\textrm{(a)}$ of Lemma \ref{Lemma - elevator}, we include a general lemma for proper self-maps of the disk.

\begin{lemma}\label{Lemma - Blaschke - diameters}
Let $P:\D \to \D$ be a proper holomorphic map of degree $D$, with $P(0)=0$, and fix $\rho\in (0,1)$. There exists a constant $C>0$ depending only on $D,\rho$ such that for each connected set $A\subset B(0,\rho)$ one has
\begin{align*}
\diam A \leq C (\diam P(A))^{1/D}.
\end{align*}
\end{lemma}
\begin{proof}
Let $ A$ be a connected subset of $\D$, and assume first that $0\in A$. Define $\zeta$ to be the furthest point of $\br {P(A)}$, so $P(A) \subset B(0,|\zeta|)$, and $\diam P(A) \simeq |\zeta|$, since $0\in P(A)$. Let $W$ be the component of $P^{-1}(B(0,|\zeta|))$ that contains $A$, and consider $\alpha$ to be the furthest point of $\br W$, so $W\subset B(0,|\alpha|)$. Using Gr\"otzsch's modulus theorem and Lemma \ref{Lemma - modulus}$\textrm{(a)}$ we have
\begin{align}\label{alpha-zeta estimates}
\mu( |\alpha|)&\geq \md (\D \setminus \br W) \geq  \frac{1}{D} \md (\D\setminus \br B(0,|\zeta|) )= \frac{1}{2\pi D} \log \frac{1}{|\zeta|}.
\end{align}
The following lemma gives us the asymptotic behavior of $\mu$ as $r\to 0$ (see \cite[pp. 72--76]{Ah}).

\begin{lemma} \label{Lemma - Grotzsch asymptotics}
There exists $r_0\in (0,1)$ such that for $0<r\leq r_0$ we have
\begin{align*}
\frac{1}{2\pi} \log\frac{2}{r}\leq  \mu(r)\leq \frac{1}{2\pi} \log \frac{8}{r}.
\end{align*}
\end{lemma}

Using this lemma, if $|\alpha|\leq r_0$, then by \eqref{alpha-zeta estimates} one has $|\alpha|\leq 8 |\zeta|^{1/D}$. If $|\alpha|>r_0$, then using \eqref{alpha-zeta estimates} and the monotonicity of $\mu$ one obtains a uniform lower bound for $|\zeta|$, thus $|\alpha|\leq 1 \lesssim |\zeta|^{1/D}$. In all cases
\begin{align}\label{alpha-zeta estimates 2}
\diam A \leq \diam W \lesssim |\alpha| \lesssim |\zeta|^{1/D} \simeq (\diam P(A))^{1/D}.
\end{align}

In the above we only assumed that $0\in A$. Now we drop this assumption, and consider a connected set $A\subset B(0,\rho)$. By Schwarz's lemma we have $|P(z)|\leq |z|$, so $P(A)\subset P( B(0,\rho))\subset \br{B}(0,\rho)$. Let $z_0\in A$, and $w_0=P(z_0)\in P(A)$. Consider  M\"obius transformations $\phi$ and $\psi$ of the disk that move $z_0$ and $w_0$ to $0$ respectively. Applying the previous case to $\psi \circ P \circ \phi^{-1}$ and the connected set $\phi(A)$ one has
\begin{align*}
\diam \phi(A) \lesssim \diam\psi(P(A))^{1/D}.
\end{align*}
However, since $A,P(A) \subset \br{B}(0,\rho)$, it follows (e.g. by direct computation using the formulas of the M\"obius transformations $\phi,\psi$) that $\diam \phi(A) \simeq \diam A$ and $\diam \psi(P(A)) \simeq \diam P(A)$ with constants depending only on $\rho$.
\end{proof}

In our case, let $P= \phi \circ f^n \circ \psi^{-1} :\D \to \D$, which is a proper map of degree bounded by $D_0$, that fixes $0$. Now, let $A\subset B$ be a connected set. Using \eqref{Koebe for psi-diamB}, and Lemma \ref{Lemma - Blaschke - diameters} applied to $\psi(A)\subset B(0,\rho)$, one has
\begin{align*}
\diam A &\simeq \diam B \cdot \diam \psi (A) \\
&\lesssim \diam B \cdot (\diam P(\psi(A)))^{1/D_0}\\
&=  \diam B \cdot (\diam \phi(f^n(A)))^{1/D_0}\\
&\simeq \diam B\cdot (\diam f^n(A))^{1/D_0},
\end{align*}
where in the end we used the fact that $\phi$ is a scaling by a fixed factor.

For the proof of part $\textrm{(b)}$, we will need again a lemma for proper maps of the disk.

\begin{lemma}\label{Lemma - Blaschke}
Let $P:\D \to \D$ be a proper holomorphic map of degree $D$, with $P(0)=0$, and fix $r_0 \in (0,1)$. Then there exists $r_1>0$ that depends only on $D,r_0$ such that $B(0,r_1)\subset P(B(0,r_0))$.
\end{lemma}
\begin{proof}
Let $B(0,r_1) \subset P(B(0,r_0))$ be a ball of maximal radius, and let $W$ be the component of $P^{-1}(B(0,r_1))$ that contains $0$. Note that $\br W$ contains a point $z$ with $|z|=r_0$. Lemma \ref{Lemma - modulus}$\textrm{(a)}$ and Gr\"otzsch's modulus theorem yield
\begin{align*}
\frac{1}{2\pi} \log \frac{1}{r_1}=\md (\D \setminus \br B(0,r_1)) \leq D\md (\D\setminus \br W) \leq D \mu (r_0).
\end{align*}
Monotonicity of $\mu$ now yields a uniform lower bound for $r_1$.
\end{proof}

In our case, Koebe's distortion theorem in \eqref{Koebe for psi-diamB} implies that $\psi(\frac{1}{2}B)$ contains a ball $B(0,r_0)$ where $r_0$ is independent of $B$. Now, Lemma \ref{Lemma - Blaschke} applied to $P= \phi \circ f^n \circ \psi^{-1}$ shows that $P( \psi(\frac{1}{2}B))$ contains some ball $B(0,r_1)$, independent of $B$. Since $\phi$ is only scaling by a certain factor, we obtain that $f^n(\frac{1}{2}B)$ contains some ball $B(x,r_2)$, independent of $B$.

Finally, we show part (c). We first need the following lemma.

\begin{lemma} \label{Lemma Blaschke Lipschitz} 
Let $P:\D \to \D$ be a proper holomorphic map of degree $D$. Then for $\rho \in (0,1)$ the restriction $P:B(0,\rho) \to \D$ is $K$-Lipschitz, where $K$ depends only on $D,\rho$.
\end{lemma}
\begin{proof}
Each proper self-map of the unit disk is a finite Blaschke product, so we can write $P(z)= e^{i\theta}\prod_{i=1}^D P_i(z)$, where $P_i(z)= (z-a_i)/(1-\br {a_i}z )$, $a_i\in \D$. Note that for $|z|<\rho$ we have
\begin{align*}
|P_i'(z)| \leq \frac{1}{|1-\br{a_i}z|^2} \leq \frac{1}{(1-\rho)^2}.
\end{align*}
Thus
\begin{align*}
|P'(z)| = \left| \sum_{i=1}^D P_i'(z) \prod_{j\neq i} P_j \right|\leq \sum_{i=1}^D |P_i'(z)|\leq \frac{D}{(1-\rho)^2}\eqqcolon K
\end{align*}
for $z\in B(0,\rho)$.
\end{proof}

For $u,v\in B$ by \eqref{existence of rho} one has $\psi(u),\psi(v)\in B(0,\rho)$. Thus, applying Lemma \ref{Lemma Blaschke Lipschitz} to $P=\phi\circ f^n \circ \psi^{-1}$, and using \eqref{Koebe for psi-diamB} we obtain
\begin{align*}
|f^n(u)-f^n(v)| &\simeq |\phi (f^n(u))-\phi(f^n(v))| = |P(\psi(u))-P(\psi(v))| \\
&\lesssim |\psi(u)-\psi(v)|\\
&\simeq \frac{|u-v|}{\diam B}.
\end{align*}
This completes the proof of Lemma \ref{Lemma - elevator}. \qed

\section{Hausdorff and Minkowski dimensions}\label{Section - Minkowski}

For a metric space $(X,d)$ and $s\in [0,\infty)$ the \textit{$s$-dimensional Hausdorff measure of $X$} is defined as
\begin{align*}
\mathcal H^s(X)= \lim_{\delta\to 0} \mathcal H^s_\delta(X),
\end{align*}
where $\mathcal H^s_\delta (X)= \inf \{ \sum_{i\in I} (\diam U_i)^s \}$ and the infimum is taken over all covers of $X$ by open sets $\{U_i\}_{i\in I}$ of diameter at most $\delta$. Then the \textit{Hausdorff dimension} of $(X,d)$ is
\begin{align*}
\dim_H X = \inf\{ s: \mathcal H^s(X) <\infty\} = \sup\{s: \mathcal H^s(X)=\infty\} \in [0,\infty].
\end{align*}

The Minkowski dimension is another useful notion of dimension for a fractal set $X\subset \R^n$. For $\varepsilon>0$ we define $n(\varepsilon)$ to be the maximal number of disjoint open balls of radii $\varepsilon>0$ centered at points $x\in X$. We then define  the upper and lower \textit{Minkowski dimensions}, respectively, as
\begin{align*}
\br{\dim}_M X&= \limsup_{\varepsilon\to 0} \frac{\log n(\varepsilon)}{\log(1/\varepsilon)},\\
\underline{\dim}_M X&= \liminf_{\varepsilon\to 0} \frac{\log n(\varepsilon)}{\log(1/\varepsilon)}.
\end{align*} 
If the two numbers agree, then we say that their common value $\dim_M X$ is the \textit{Minkowski}, or else, \textit{box dimension} of $X$.

It is easy to see that the definition of the Minkowski dimension is not affected if $n(\varepsilon)$ denotes instead the smallest number of open balls of radii $\varepsilon>0$ centered at $X$, that cover $X$. The important difference between the Hausdorff and Minkowski dimensions is that in the Hausdorff dimension we are taking into account coverings $\{U_i\}_{i\in I}$ with different weights $(\diam U_i)^s $ attached to each set, but in the Minkowski dimension we are considering only coverings of sets with equal diameters. It easily follows from the definitions that we always have
\begin{align*}
\dim_H X\leq \underline{\dim}_M X \leq \br{\dim}_M X.
\end{align*}

From now on, $n(\varepsilon)$ will denote the maximal number of disjoint open balls of radii $\varepsilon$, centered at points $x\in X$. Based on the distortion estimates that we developed in Section \ref{Section - Conformal elevator}, and using results of \cite{Fal} and \cite{SU} we have the following result that concerns the Hausdorff and Minkowski dimensions of Julia sets of semi-hyperbolic maps.
\begin{theorem}\label{Theorem - Hausdorff measure positive} Let $f:\widehat \C \to \widehat \C$ be a semi-hyperbolic rational map with $\mathcal J(f)\neq \widehat{\C}$ and $s\coloneqq \dim_H \J(f)$. We have 
\begin{enumerate}[label=\emph{(\alph*)}]
\item $0<\mathcal H^s (\J(f))<\infty$
\item $ \dim_M \J(f)= s= \dim_H \J(f)$
\item There exists a constant $C>0$ such that for all $\varepsilon>0$
$$ \frac{1}{C}\leq  n(\varepsilon) \varepsilon ^s\leq C,$$
\end{enumerate}
where $n(\varepsilon)$ is the maximal number of disjoint open balls  of radii $\varepsilon$ (in the spherical metric), centered in $\J(f)$.
\end{theorem}
\begin{proof} By considerations as in the beginning of Section \ref{Section - Conformal elevator}, we may assume that $\J(f) \subset \D$, and use the Euclidean metric which is comparable to the spherical metric. This will only affect the constant in part (c) of the theorem. 

The parts (a) and (b) follow from \cite[Theorem 1.11(e) and (g)]{SU}. Also, if $B_1,\dots,B_{n(\varepsilon)}$ are disjoint balls of radius $\varepsilon>0$ centered at $\J(f)$ then the collection $2B_1,\dots,2B_{n(\varepsilon)}$ covers $\J(f)$, where $2B_i$ has the same center as $B_i$ but twice the radius. Thus, we have
\begin{align*}
\mathcal H_{\varepsilon}^s (\J(f)) \leq n(\varepsilon)( 2\varepsilon)^s.
\end{align*} 
Taking limits, and using (a), we obtain 
$$0<\mathcal H^s(\J(f)) \leq 2^s \liminf_{\varepsilon\to 0} n(\varepsilon) \varepsilon^s$$
which shows the left inequality in (c).

For the right inequality in (c), we use the following result of Falconer.
\begin{theorem}[{{\cite[Theorem 4]{Fal}}}]\label{Theorem - Falconer}
Let $(F,d)$ be a compact metric space with $s=\dim_H F <\infty$. Suppose that there exist $K_0,r_0>0$ such that for any ball $B \subset F$ of radius $r<r_0$ there is a mapping $\psi :F\to B$ satisfying 
\begin{align*}
K_0 r \cdot d(x,y) \leq d(\psi(x),\psi(y))
\end{align*}
for all $x,y\in F$. Then $\limsup_{\varepsilon\to 0}n(\varepsilon) \varepsilon^s <\infty$.
\end{theorem}
We remark that the mapping $\psi:F\to B$ need not be continuous.

It remains to show that this theorem applies in our case. To show the existence of $\psi$ we will carefully use the distortion estimates of Lemma \ref{Lemma - elevator}. Let $r_0$ be so small that for $r<r_0$ and $p\in \J(f)$ the conclusions of Lemma \ref{Lemma - elevator} are true for the ball $B=B(p,r)$. In particular, there exists $r_1$, independent of $B$, such that
\begin{align}\label{Theorem - Falc-r_1}
B(f^n(p),r_1)\subset f^n(B)  
\end{align}
for some $n\in \N$. 

For each ball $B(q,r_1), q\in \J(f)$, there exists $m\in \N$ such that $f^m: B(q,r_1) \cap \J(f) \to \J(f)$ is surjective (e.g. see \cite[Corollary 14.2]{Mil}). We choose the smallest such $m$. Compactness of $\J(f)$ allows us to choose a uniform $m\in \N$, independent of $q\in\J(f)$. By the analyticity of $f^m$, there exists a constant $K_1>0$ such that for all $u,v\in \J(f)$ we have
\begin{align}\label{Theorem - Falc- f^m}
|f^m(u)-f^m(v)| \leq K_1 |u-v|.
\end{align}
Also, by Lemma \ref{Lemma - elevator}(c), there exists $K_2$ independent of $B$ such that
\begin{align}\label{Theorem - Falc- f^n}
|f^n(u)-f^n(v)|\leq K_2 \frac{|u-v|}{\diam B}
\end{align}
for $u,v\in B$. Here $n$ depends on the ball $B$ and is defined as in the comments preceding Lemma \ref{Lemma - elevator}.

Now, we can construct the desired $\psi: \J(f) \to B$. Let $g:\J(f) \to B(f^n(p), r_1) \cap \J(f)$ be any right inverse of the surjective map $f^m: B(f^n(p),r_1) \cap \J(f) \to \J(f)$. Also, the inclusion \eqref{Theorem - Falc-r_1} allows us to define a right inverse $h: B(f^n(p),r_1)\cap \J(f) \to B\cap \J(f)$ of $f^n$, restricted on a suitable subset of $B$. Now, let $\psi= h\circ g :\J(f) \to B\cap \J(f)$, and observe that by \eqref{Theorem - Falc- f^n}, and \eqref{Theorem - Falc- f^m} we have
\begin{align*}
|\psi(u)-\psi(v)|&= | h (g (u))-h ( g(v))| \\
&\geq \frac{\diam B}{K_2} |g(u)-g(v)|\\
&\geq \frac{\diam B}{K_1K_2} |u-v| \\
&= \frac{2}{K_1K_2} r|u-v|.
\end{align*}
Thus, the hypotheses of Theorem \ref{Theorem - Falconer} are satisfied with $K_0= 2/(K_1K_2)$.
\end{proof}

\section{Homogeneous sets and Julia sets}\label{Section - homogeneous}
Let $\mathcal P =\{C_k\}_{k\geq 0}$ be a packing, as defined in the Introduction, where $C_k$ are topological circles, surrounding topological open disks $D_k$ (in the plane or the sphere) such that $D_0 $ contains $C_k$ for $k\geq 1$, and ${D_k}, k\geq 1$ are disjoint. Then the set $\mathcal S= \br {D_0} \setminus \bigcup_{k\geq 1} D_k$ is the residual set $\mathcal S$ of the packing $\mathcal P$. 

In the following, one can use the Euclidean or spherical metric, but it is convenient to consider $C_0=\partial D_0$ as the boundary of the unbounded component of the packing $\mathcal P$ (see Figures \ref{fig:apollonian} and \ref{fig:test2}), and use the Euclidean metric to study the other disks $D_k,k\geq 1$. Thus, we will restrict ourselves to the use of the Euclidean metric in this section.

Following \cite{MS}, we say that the residual set $\mathcal S$ is \textit{homogeneous} if it satisfies properties $(1),(2)$ and $(3)$, or $(1),(2)$ and $(4)$ below. 
\begin{enumerate}
\item Each $D_k,k\geq 1$ is a  \textit{uniform quasi-ball}. More precisely, there exists a constant $\alpha \geq 1$ such that for each $D_k$ there exist inscribed and circumscribed, concentric circles of radii $r_k$ and $R_k$ respectively with 
\begin{align*}
\frac{R_k}{r_k} \leq \alpha.
\end{align*}
\item There exists a constant $\beta\geq 1$ such that for each $p\in \mathcal S$ and $0<r\leq \diam \mathcal S$ there exists a circle $C_k$ intersecting $B(p,r)$ such that 
\begin{align*}
\frac{1}{\beta} r\leq \diam C_k \leq \beta r.
\end{align*}
\item The circles $C_k$ are \textit{uniformly relatively separated}. This means that there exists $\delta>0$ such that 
\begin{align*}
\Delta(C_j,C_k)\coloneqq \frac{\dist(C_j,C_k)}{\min \{\diam C_j, \diam C_k \}} \geq \delta
\end{align*}
for all $j\neq k$.
\item The disks $D_k,k\geq 1$ are \textit{uniformly fat}. By definition, this means that there exists $\tau>0$ such that for every ball $B(p,r)$ centered at $D_k$ that does not contain $D_k$, we have
\begin{align*}
m_2 (D_k \cap B(p,r)) \geq \tau r^2,
\end{align*}
where $m_2$ denotes the $2$-dimensional Lebesgue measure. (Here one can use the spherical measure for packings on the sphere.)
\end{enumerate} 

Condition $(1)$ means that the sets $D_k$ look like round balls, while $(2)$ says that the circles $C_k$ exist in all scales and all locations in $\mathcal S$. Condition $(3)$ forbids two ``large" circles $C_k$ to be close to each other in some uniform manner. Note that this only makes sense when $\br {D_k}, k\geq 1$ are disjoint, e.g. in the case of a Sierpi\'nski carpet. Finally, $(4)$ is used to replace $(3)$ when we are working with fractals such as the Sierpi\'nski gasket, or generic Julia sets regarded as packings, where $\br {D_k}$ are not disjoint. We now summarize some interesting properties of homogeneous sets, that are not needed though for the proof of Theorem \ref{Theorem - Main}. 

A set $E\subset \R^n$ is said to be \textit{porous} if there exists a constant $0<\eta<1$ such that for all sufficiently small $r>0$ and all $x\in E$, there exists a point $y\in \R^n$ such that 
\begin{align*}
B(y,\eta r) \subset B(x,r)\setminus E.
\end{align*}
A Jordan curve $\gamma \subset \C$ is called a $K$-\textit{quasicircle} if for all $x,y\in \gamma$ there exists a subarc $\gamma_0 $ of $\gamma$ joining $x$ and $y$ with $\diam \gamma_0 \leq K |x-y|$. The \textit{(Ahlfors regular) conformal dimension} of a metric space $(X,d)$, denoted by $\textrm{(AR)}\Cdim X$, is the infimum of the Hausdorff dimensions among all (Ahlfors regular) metric spaces that are quasisymmetrically equivalent to $(X,d)$. For more background see Chapters $10$ and $15$ in \cite{He}.    

\begin{prop}\label{Proposition - Homogeneous}
Let $\mathcal S$ be the residual set of a packing $\mathcal P$, satisfying $(1)$ and $(2)$. Then
\begin{enumerate}[label=\emph{(\alph*)}]
\item $\mathcal S$ is locally connected.
\item $\mathcal S$ is porous.
\item $\dim_H\mathcal S \leq 2- \delta$, where $\delta>0$ depends only on the constants in $(1),(2)$.
\end{enumerate} 
Furthermore, if instead of $(1)$ and $(2)$ we only assume that $\mathcal S$ satisfies $(3)$ and the topological circles $C_k=\partial D_k$ are uniform quasicircles, then $\Cdim \mathcal S >1$.
\end{prop}
\begin{proof}
By $(1)$, each $D_k$ contains a ball of diameter comparable to $\diam D_k$. Thus, summing the areas of the sets $D_k$, and noting that they are all contained in $\br {D_0}\subset \C$, we see that for each $\varepsilon>0$, there can only be finitely many sets $D_k$ with $\diam D_k>\varepsilon$. We conclude that $\mathcal S$ is locally connected (see \cite[Lemma 19.5]{Mil}).

Condition $(2)$ implies that for $r\leq \diam \mathcal S$, every ball $B(p,r)$ centered at $\mathcal S$ intersects a curve $C_k$ of diameter comparable to $r$. Let $c<1$ and consider the ball $B(p,cr)\subset B(p,r)$. Then $B(p,cr)$ intersects a curve $C_k$ of diameter comparable to $cr$, and if $c$ is sufficiently small but uniform, then $C_k\subset B(p,r)$. Thus $B(p,r)$ contains a curve $C_k$ of diameter comparable to $r$. By $(1)$, $D_k$ contains a ball of radius comparable to $\diam D_k$ and thus comparable to $r$ (note that here we use the Euclidean metric). Hence, $B(p,r)\setminus \mathcal S$ contains a ball of radius comparable to $r$. This completes the proof that $\mathcal S$ is porous. 

It is a standard fact that a porous set $E\subset \R^n$ has Hausdorff dimension bounded away from $n$, quantitatively (see \cite[Theorem 3.2]{Sa}). Thus, (b) implies (c).

For our last assertion we will use a criterion of Mackay \cite[Theorem 1.1]{Mac} which asserts that a doubling metric space which is {annularly linearly connected} has conformal dimension strictly greater than $1$. A connected metric space $X$ is \textit{annularly linearly connected} (abbr. ALC) if there exists some $L\geq 1$ such that for every $p\in X,r>0$, and $x,y\in X$ in the annulus $A(p,r,2r)\coloneqq \br B(p,2r)\setminus B(p,r)$ there exists an arc $J\subset X$ joining $x$ to $y$ that lies in a slightly larger annulus $A(p,r/L,2Lr)$.

It suffices to show that $\mathcal S$ is ALC. The idea is simple, but the proof is technical, so we only provide a sketch. Let $x,y\in A(p,r,2r)\cap \mathcal S$, and consider a path $\gamma \subset A(p,r,2r)$ (not necessarily in $\mathcal S$) that joins $x$ and $y$. The idea is to replace the parts of the path $\gamma$ that lie in the complementary components $D_k$ of $\mathcal S$ by arcs in $C_k=\partial D_k$ and then make sure that the resulting arc stays in a slightly larger annulus $A(p,r/L,2Lr)$. The assumption that the curves $C_k$ are quasicircles guarantees that the subarcs that we will use are not too ``large", and condition $(3)$ guarantees that the ``large" curves $C_k$ do not block the way from $x $ to $y$, since these curves are not allowed to be very close to each other.

Using $(3)$, we can find uniform constants $a,L_1\geq 1$ such that there exists at most one curve $C_{k_0}$ with $\diam C_{k_0} \geq  r/a$ that intersects $B(p, r/L_1)$. We call a curve $C_k$ \textit{large} if its diameter exceeds $r/a$, and otherwise we call it \textit{small}. We enlarge slightly the annulus (maybe using a larger $L_1$) to an annulus $A(p,r/L_1,2rL_1)$ so that $B(p,2rL_1)$ contains all small curves $C_k$ that intersect $\gamma$. We now check all different cases. 

If $\gamma$ meets the large $C_{k_0}$ that intersects $B(p,r/L_1)$, using the fact that $C_{k_0}$ is a quasicircle, we can enlarge the annulus to an annulus $A(p,r/L_2,2rL_2)$ with a uniform $L_2\geq 1$, so that $x$ can be connected to $y$ by a path in $A(p,r/L_2,2rL_2)\setminus D_{k_0}$. We call the resulting path $\gamma$. Note that here we have to assume that $C_{k_0}\neq C_0$, so that the path $\gamma$ does not lie in the unbounded component of the packing and it passes through several curves $C_k$ on the way from $x$ to $y$. The case $C_{k_0}=C_0$, which occurs only when $x,y\in \br {D_0}$, is similar and in the previous argument we just have to choose a path $\gamma$ that lies in $A(p,r/L_2,2rL_2)\cap  \br {D_0}$. We still assume that $B(p,2rL_2)$ contains all small curves $C_k$ that intersect $\gamma$.

If $\gamma$ meets a small $C_k$ that does not intersect $B(p,r/L_2)$, then we can replace the subarcs of $\gamma$ that lie in $D_k$ with arcs in $C_k$ that have the same endpoints. The resulting arcs will lie in the annulus by construction. Next, if $\gamma$ meets a small $C_k$ that does intersect $B(p,r/L_2)$, we follow the same procedure as before, but now we have to choose the sub-arcs of $C_k$ carefully, so that they do not approach $p$ too much. This can be done using the assumption that the curves $C_k$ are uniform quasicircles. The resulting arcs will lie in a slightly larger annulus $A(p,r/L_3, 2rL_3)$, where $L_3\geq 1$ is a uniform constant.

Finally, if $\gamma$ intersects a large $C_k$ which does not meet $B(p,r/L_3)$ we can use the assumption that $C_k$ is a quasicircle to replace the subarcs of $\gamma$ that lie in $D_k$ with subarcs of $C_k$ that have diameter comparable to $r$. Thus, a larger annulus $A(p,r/L_4,2rL_4)$ will contain the arcs of $C_k$ that we obtain in this way.

We need to ensure that this procedure indeed yields a path that joins $x$ and $y$ inside $A(p,r/L_4,2rL_4)$. This follows from the fact that $\diam D_k \to 0$. The latter fact follows from the assumption that the curves $C_k$ are uniform quasicircles, which in turn implies that each $D_k$ contains a ball of radius comparable to $\diam D_k$, i.e., $(1)$ is true (for a proof of this assertion see \cite[Proposition 4.3]{Bo}).
\end{proof}

Next, we continue our preparation for the proof of Theorem \ref{Theorem - Main}.

From now on, we will be using a slightly more general definition for a packing $\mathcal P=\{ C_k\}_{k\geq 0}$, suitable for Julia sets, where the sets $D_k$ are allowed to be simply connected open sets and $C_k=\partial D_k$ (so they are not necessarily topological circles). Making abuse of terminology, we still call $C_k$ a ``curve".

As we will see in Lemma \ref{Lemma - N - n}, a homogeneous set has the special property that there is some important relation between the curvature distribution function $N(x)$ and the maximal number of disjoint open balls $n(\varepsilon)$, centered at $\mathcal S$. Thus, considerations about the residual set $\mathcal S$, which are reflected by $n(\varepsilon)$, can be turned into considerations about the complementary components $D_k$, which are comprised in $N(x)$.

The following lemma is proved in \cite{MS} and its proof is based on area and counting arguments. 

\begin{lemma}[{{\cite[Lemma 3]{MS}}}]\label{Lemma - Mereknov - Sabitova} 
Assume that $\mathcal S$ is the residual set of a packing $\mathcal P= \{C_k\}_{k\geq 0}$ that satisfies $(1)$ and $(3)$ (or $(1)$ and $(4)$). For any $\beta>0$, there exist constants $\gamma_1, \gamma_2>0$ depending only on $\beta$ and the constants in $(1), (3)$ (or $(1), (4)$) such that for any collection $\mathcal C$ of disjoint open balls of radii $r>0$ centered in $\mathcal S$ we have the following statements:
\begin{enumerate}[label=\emph{(\alph*)}]
\item\label{Lemma - M-S-1}{There are at most $\gamma_1$ balls in $\mathcal C$ that intersect any given $C_k$ with }\\\notag
$$\diam C_k \leq \beta r.$$
\item \label{Lemma - M-S-2}{ There are at most $\gamma_2$ curves $C_k$ intersecting any given ball in $2\mathcal C$ and satisfying}\\\notag
$$\frac{1}{\beta} r\leq \diam C_k, $$
\textrm{where $2\mathcal C$ denotes the collection of open balls with the same centers as the ones in $\mathcal C$, but with radii $2r$. }
\end{enumerate}
\end{lemma}

Using this lemma one can prove a relation between the curvature distribution function $N(x)= \# \{ k: (\diam C_k)^{-1} \leq x\}$ (using the Euclidean metric) and the maximal number $n(\varepsilon)$ of disjoint open balls of radius $\varepsilon$, centered at $\mathcal S$. Namely, we have the following lemma.

\begin{lemma}\label{Lemma - N - n} 
Assume that the residual set $\mathcal S$ of a packing $\mathcal P$ satisfies $(1),(2)$ and $(3)$ or $(1),(2)$ and $(4)$. Then there exists a constant $C>0$ such that for all small $\varepsilon>0$ we have
\begin{align*}
\frac{1}{C} n(\varepsilon) \leq N(\beta/\varepsilon) \leq C n(\varepsilon), 
\end{align*}
where $\beta$ is the constant in $(2)$. 
\end{lemma}

The proof is essentially included in the proof of \cite[Proposition 2]{MS} but we include it here for completeness.

\begin{proof} Let $\mathcal C$ be a maximal collection of disjoint open balls of radius $\varepsilon$, centered at $\mathcal S$. For each ball $C\in \mathcal C$, by condition $(2)$ there exists $C_k$ such that $C_k\cap C \neq \emptyset $ and $\frac{1}{\beta} \varepsilon \leq \diam C_k \leq \beta \varepsilon $. On the other hand, Lemma \ref{Lemma - Mereknov - Sabitova}(a) implies that for each such $C_k$ there exist at most $\gamma_1$ balls in $\mathcal C$ that intersect it. Thus
\begin{align*}
n(\varepsilon)= \# \mathcal C \leq \gamma_1\cdot  \# \biggl\{ k:  \frac{1}{\beta} \varepsilon \leq \diam C_k \leq \beta \varepsilon \biggr\} \leq \gamma_1 N(\beta/\varepsilon).
\end{align*}

Conversely, note that by the maximality of $\mathcal C$, it follows that $2\mathcal C$ covers $\mathcal S$. Hence, if $C_k$ is arbitrary satisfying $\diam C_k\geq \frac{1}{\beta}\varepsilon$, it intersects a ball $2C$ in $2\mathcal C$. For each such ball $2C$, Lemma \ref{Lemma - Mereknov - Sabitova}(b) implies that there exist at most $\gamma_2$ curves $C_k$ with $\diam C_k\geq \frac{1}{\beta}\varepsilon$ that intersect it. Thus
\[
N(\beta/\varepsilon)= \#\biggl\{k: \diam C_k\geq \frac{1}{\beta}\varepsilon\biggr\} \leq \gamma_2  \cdot \# 2\mathcal C = \gamma_2 n(\varepsilon).\qedhere
\]
\end{proof}

Finally, we proceed to the proofs of Theorem \ref{Theorem - Main} and Corollary \ref{Corollary}.

\begin{proof}[Proof of Theorem \ref{Theorem - Main}]
By considerations as in the beginning of Section \ref{Section - Conformal elevator}, we assume that $\J(f)\subset \D$, and we will use the Euclidean metric since this does not affect the conclusion of the Theorem. Let $C_0$ be the boundary of the unbounded Fatou component, $D_k,k\geq 1$ be the sequence of bounded Fatou components, and $C_k=\partial D_k$. Then $\mathcal P=\{C_k\}_{k\geq 0}$ can be viewed as a packing, and $\mathcal S=\J(f)$ is its residual set. Note, though, that the sets $C_k$ need not be topological circles in general, as we already remarked. This, however, does not affect our considerations, since it does not affect the conclusions of lemmas \ref{Lemma - Mereknov - Sabitova} and \ref{Lemma - N - n}, as long as the other assumptions hold for $C_k$ and the simply connected regions $D_k$ enclosed by them. We will freely use the terminology ``curves" for the sets $C_k$.

By Theorem \ref{Theorem - Hausdorff measure positive} we have that the quantity $n(\varepsilon)\varepsilon ^s$ is bounded away from $0$ and $\infty$ as $\varepsilon\to 0$, where $s=\dim_H \J(f)$. If we prove that $\J(f)$ is a homogeneous set, satisfying $(1),(2)$ and $(4)$, then using Lemma \ref{Lemma - N - n}, it will follow that $N(x)/x^s$ is bounded away from $0$ and $\infty$ as $x\to\infty$, and in particular
\begin{align*}
0<\liminf_{x\to\infty} \frac{ N(x)}{x^s} \leq \limsup_{x\to\infty} \frac{N(x)}{x^s} <\infty
\end{align*} 
which will complete the proof.

Julia sets of semi-hyperbolic rational maps are locally connected if they are connected (see \cite[Theorem 1.2]{Yin} and also \cite[Proposition 10]{Mih}), and thus for each $\varepsilon>0$ there exist finitely many Fatou components with diameter greater than $\varepsilon$ (see \cite[Theorem 4.4, pp.~112--113]{Wh2}). 

First we show that condition $(1)$ in the definition of homogeneity is satisfied. The idea is that the finitely many large Fatou components are trivially quasi-balls, as required in $(1)$, so there is nothing to prove here, but the small Fatou components can be blown up with good control to the large ones using Lemma \ref{Lemma - elevator}. The distortion estimates allow us to control the size of inscribed circles of the small Fatou components.
   
Let $d_0 \leq  (1/4K_1)^{1/\gamma}$, where $K_1,\gamma$ are the constants appearing in Lemma \ref{Lemma - elevator}. We also make $d_0$ even smaller so that for $r\leq d_0$ and $p\in \J(f)$ the conclusions of Lemma \ref{Lemma - elevator} are true. Since there are finitely many curves $C_k$ with $\diam C_k > d_0/2$, for these $C_k$ there exist concentric inscribed and circumscribed circles with radii $r_k$ and $R_k$ respectively, such that $R_k/r_k \leq \alpha$, for some $\alpha>0$. This implies that $2r_k \leq \diam C_k \leq 2R_k \leq 2\alpha r_k$. 

If $C_k$ is arbitrary with $\diam C_k \leq d_0/2$, then for $p\in C_k$ and $r=2\diam C_k$, by Lemma \ref{Lemma - elevator}(a) there exists $n\in \N$ such that 
\begin{align*}
\frac{r/2}{2r}=\frac{\diam C_k}{\diam B(p,r)} \leq K_1 ( \diam f^n(C_k))^\gamma.
\end{align*}  
Note that the Fatou component $D_k$ is mapped under $f^n$ onto a Fatou component $D_k'$. Since $f^n$ is proper, the boundary $C_k$ of $D_k$ is mapped onto $C_k'\coloneqq\partial D_k'$. Then the above inequality can be written as
\begin{align*}
\diam C_k' \geq (1/4K_1)^{1/\gamma} \geq d_0.
\end{align*}
Hence, $C_k'$ is one of the ``large" curves, for which there exists a inscribed ball $B(q',r_k')$ such that $2r_k'\leq \diam C_k' \leq 2\alpha r_k'$. Observe that $r_k'\geq d_0/2\alpha$.

Let $q \in D_k\subset B(p,r)$ be a preimage of $q'$ under $f^n$, and $W\subset D_k$ be the component of $f^{-n} (B(q',r_k'))$ that contains $q$. For each $u\in \partial W$, by Lemma \ref{Lemma - elevator}(c) one has
\begin{align*}
r_k'=|f^n(q)-f^n(u)|\leq K_2 \frac{|q-u|}{2r}.
\end{align*}
Thus
\begin{align*}
\frac{\diam C_k} {|q-u|} =\frac{r/2}{|q-u|} \leq \frac{K_2}{4r_k'}  \leq \frac{\alpha K_2}{2d_0}.
\end{align*}
Letting $R_k=\diam C_k$, and $r_k= \inf_{u\in \partial W} |q-w|$, one obtains $R_k/r_k \leq \alpha K_2/2d_0$, so $(1)$ is satisfied with $\alpha'= \max\{\alpha, \alpha K_2/ 2d_0 \}$.

Similarly, we show that condition $(2)$ is also true. Let $r_1$ be the constant in Lemma \ref{Lemma - elevator}(b) and consider $d_0\leq r_1/2$ so small that the conclusions of Lemma \ref{Lemma - elevator} are true for $p\in \J(f)$ and $r\leq d_0$. Note that by compactness of $\J(f)$ there exists $\beta>0$ such that for $d_0\leq r\leq \diam \J(f)$ and $p\in \J(f)$ there exists $C_k$ such that $C_k\cap B(p,r)\neq \emptyset$ and
\begin{align}\label{Proof - (2)}
\frac{1}{\beta} r \leq  \diam C_k \leq \beta r.
\end{align}
Indeed, one can cover $\J(f)$ with finitely many balls $B_1,\dots,B_{N}$ of radius $d_0/2$ centered at $\J(f)$, such that each ball $B_j$ contains a curve $C_{k(j)}$. This is possible because every ball $B_j$ centered in the Julia set must intersect infinitely many Fatou components, otherwise $f^n$ would be a normal family in $B_j$. In particular, by local connectivity ``most" Fatou components are small, and thus one of them, say $D_{k(j)}$, will be contained in $B_j$. Now, if $B(p,r)$ is arbitrary with $p\in \J(f),r\geq d_0$, we have that $p\in B_j$ for some $j\in\{1,\dots,N\}$, and thus $ B_j\subset B(p,r)$. Since $r\in [d_0,\diam \J(f)]$ lies in a compact interval, \eqref{Proof - (2)} easily follows, by always using the same finite set of curves $C_{k(1)},\dots,C_{k(N)}$ that correspond to $B_1,\dots,B_N$, respectively. We may also assume that $\diam C_{k(j)} <r_1 /2$ for each of these curves.

Now, if $r<d_0$, $p\in \J(f)$, by Lemma \ref{Lemma - elevator}(b) we have $B(f^n(p),r_1)\subset f^n(B(p,r))$ for some $n\in \N$. By the previous, $B(f^n(p),r_1/2)$ intersects some $C_k'=C_{k(j)}$ with $\diam C_k' <r_1/2$, thus $C_k' \subset B(f^n(p),r_1)$. Hence, $B(p,r)$ contains a preimage $C_k$ of $C_k'$, and by Lemma \ref{Lemma - elevator}(a), (c) we obtain
\begin{align*}
 \frac{1}{K_2} \diam C_k' \leq  \frac{\diam C_k}{\diam B(p,r)} \leq K_1 (\diam C_k')^{\gamma}.
\end{align*}
However, $C_k'$ was one of the finitely many curves that we chose in the previous paragraph. This and the above inequalities impliy that $\diam C_k \simeq \diam B(p,r) =2r$ with uniform constants. This completes the proof of $(2)$.

Finally, we will prove that condition $(4)$ of homogeneity is satisfied. This follows easily from the fact that the Fatou components of a semi-hyperbolic rational map are \textit{uniform John domains} in the spherical metric \cite[Proposition 9]{Mih}. Since we are only interested in the bounded Fatou components, we can use instead the Euclidean metric. A domain $\Omega\subset \C$ is a $\lambda$-John domain ($0<\lambda \leq 1$) if there exists a basepoint $z_0\in \Omega$ such that for all $z_1\in \Omega$ there exists an arc $\gamma\subset \Omega$ connecting $z_1$ to $z_0$ such that for all $z\in \gamma$ we have
\begin{align*}
\delta(z) \geq \lambda |z-z_1|,
\end{align*}
where $\delta(z)\coloneqq \dist(z,\partial \Omega)$.

In our case, the bounded Fatou components $D_k, k\geq 1$ are uniform John domains, i.e., John domains with the same constant $\lambda \leq 1 $. Let $B(p,r)$ be a ball centered at some $D_k$ that does not contain $D_k$. We will show that there exists a uniform constant $\tau>0$ such that 
\begin{align}\label{Proof - (4)}
m_2( D_k\cap B(p,r))\geq \tau r^2.
\end{align} 
If $B(p,r/2) \subset D_k$, then $m_2( D_k \cap B(p,r)) \geq m_2(B(p,r/2)) = \pi r^2/4$, so \eqref{Proof - (4)} is true with $\tau= \pi/4$. Otherwise, $\partial B(p,r/2)$ intersects $C_k=\partial D_k$ at a point $z_1$. We split in two cases:
 
\begin{case} The basepoint $z_0$ satisfies $|z_0-p|\geq r/4$. Then consider a path $\gamma\subset D_k$ from $z_0$ to $p$, as in the definition of a John domain, such that $\delta(z) \geq \lambda |z-p|$ for all $z\in \gamma$. In particular let $z_2 \in \gamma$ be a point such that $|z_2-p|=r/4$, thus $\delta(z_2)\geq \lambda r/4 $. Since $\lambda\leq 1$, we have $B(z_2,\lambda r/4) \subset D_k\cap B(p,r/2) $, hence 
$$m_2(D_k\cap B(p,r))\geq m_2(B(z_2, \lambda r/4))= \pi \lambda^2 r^2 /4.$$
\end{case}

\begin{case} The basepoint $z_0$ lies in $B(p,r/4)$. Consider a point $z_3\in D_k$ close to $z_1\in \partial D_k\cap \partial B(p,r/2)$ such that $|z_0-z_3|\geq r/4$. Then, by the definition of a John domain for $z=z_0$, we have $\delta(z_0)\geq \lambda |z_0-z_3|\geq \lambda r/4$. Hence, $B(z_0, \lambda r/4) \subset D_k\cap B(p,r/2)$, so 
$$m_2(D_k\cap B(p,r)) \geq \pi \lambda^2 r^2/4.$$
\end{case} 
Summarizing, $(4)$ is true for $\tau= \pi \lambda^2/4$.
\end{proof}

\begin{remark} Even when the Julia set of a semi-hyperbolic map is a Sierpi\'nski carpet, the uniform relative separation of the peripheral circles $C_k$ in condition $(3)$ need not be true. In fact, it is known that for such Julia sets condition $(3)$ is true if and only if for all critical points $c\in \J(f)$, $\omega(c)$ does not intersect the boundary of any Fatou component; see \cite[Proposition 3.9]{QYZ}. Recall that $\omega(c)$ is the set of accumulation points of the orbit $\{f^n(c)\}_{n\in \N}$.
\end{remark} 

\begin{remark} In \cite[Proposition 3.7]{QYZ} it is shown that if the boundaries of Fatou components of a semi-hyperbolic map $f$ are Jordan curves, then they are actually uniform quasicircles. If, in addition, they are uniformly relatively separated (i.e., condition $(3)$), Proposition \ref{Proposition - Homogeneous} implies that $\textrm{AR}\Cdim \J(f) \geq \Cdim \J(f)>1$. 

On the other hand, if $f$ is a semi-hyperbolic \textit{polynomial} with connected Julia set, then not all boundaries of Fatou components are Jordan curves. In fact, $\J(f)$ coincides with the boundary of a single Fatou component $\mathcal A$ which is a John domain, and is called the \textit{basin of attraction of $\infty$}; \cite[Theorem 1.1]{CJY}. According to a recent result of Kinneberg (\cite[Theorem 1.1]{Ki}), which is based on \cite[Theorem 1.2]{Car}, boundaries of planar John domains have Ahlfors regular conformal dimension equal to $1$, if they are connected. Therefore, $\textrm{AR}\Cdim  \J(f)=1$, in contrast to the previous case.  
\end{remark}

Next, we prove Corollary \ref{Corollary}.

\begin{proof}[Proof of Corollary \ref{Corollary}]

Let $s=\dim_H \J(f)$. By Theorem \ref{Theorem - Main} there exists a constant $C>0$ such that 
\begin{align}\label{Corollary - proof}
\frac{1}{C} x^s \leq N(x) \leq C x^s
\end{align}
for all $x>0$. Taking logarithms, one obtains
$$ \frac{\log(1/C)}{\log x}+ s\leq \frac{\log N(x)}{\log x} \leq s + \frac{\log C}{\log x}.$$
Letting $x\to \infty$ yields $\lim_{x\to\infty} \log N(x)/\log x= s$ which completes part of the proof.

Recall that the exponent $E$ of the packing of the Fatou components of $f$ is defined by 
$$E= \inf \biggl\{ t\in \R : \sum_{k\geq 0} (\diam C_k)^t <\infty \biggr\}$$
and it remains to show that $E=s$. Note that for $t=0$ the sum $E(t)\coloneqq\sum_{k\geq 0} (\diam C_k)^t$ diverges. Also, since for semi-hyperbolic rational maps there are only finitely many ``large" Fatou components, if $E(t_0)=\infty$, then $E(t)=\infty$ for all $t\leq t_0$. If $t<s$, using \eqref{Corollary - proof}, one has
\begin{align*}
\sum_{k\geq 0} (\diam C_k)^t &= \lim_{n\to\infty} \sum_{\substack{k \geq 0 \\ \diam C_k \geq 1/n}} (\diam C_k)^t \\
&\geq \liminf_{n\to\infty} \frac{1}{n^t} N(n)\\
& \geq  \liminf_{n\to\infty} \frac{1}{C} n^{s-t}\\
& =\infty.  
\end{align*}
This implies that $E\geq s$.

Conversely, assume that $t>s$. Since there are only finitely many ``large'' Fatou components, we only need to take into account the sets $C_k$ with $\diam C_k\leq 1$ in the sum $\sum (\diam C_k)^t$. Using again \eqref{Corollary - proof} we have
\begin{align*}
\sum_{ \substack{k\geq 0 \\ \diam C_k\leq 1}} (\diam C_k)^t &= \sum_{n=1}^\infty  \sum_{ \substack{k\geq 0 \\ 1/2^{n}< \diam C_k \leq {1}/{2^{n-1}}} } (\diam C_k)^t \\
&\leq  \sum_{n=1}^\infty  \frac{1}{2^{(n-1)t}} N(2^n)\\
&\leq C2^t\sum_{n=1}^\infty \frac{1}{2^{n(t-s)}} \\
&<\infty.
\end{align*}
Hence $E\leq s$, which completes the proof.
\end{proof}

\end{document}